\documentclass[10pt]{amsart}

\usepackage{amssymb,amsmath,amsthm}
\usepackage{tikz}
\usepackage{tikz-cd}
\usepackage[arc,all]{xy}
\usepackage{mathrsfs}
\usepackage{color}

\usepackage{comment}
\usepackage[colorlinks=true,linkcolor=blue,urlcolor=black,bookmarks=false,bookmarksopen=false]{hyperref}

\theoremstyle{plain}
\newtheorem{cor}[subsection]{Corollary}
\newtheorem{lem}[subsection]{Lemma}
\newtheorem{thm}{Theorem}
\newtheorem{prop}[subsection]{Proposition}

\theoremstyle{definition}
\newtheorem{defn}[subsection]{Definition}

\theoremstyle{remark}
\newtheorem{rem}[subsection]{Remark}

\newcommand{\Z}{\mathbb{Z}}
\newcommand{\R}{\mathbb{R}}
\newcommand{\C}{\mathbb{C}}
\newcommand{\KK}{\mathbb{K}}
\newcommand{\opn}{\operatorname}
\newcommand{\ZZ}{\mathcal{Z}}
\newcommand{\K}{\mathcal{K}}
\newcommand{\LL}{\mathcal{L}}

\numberwithin{equation}{section}

\title{On the rank of the double cohomology of moment-angle complexes}

\author{Zhilei Zhang}

\address{School of Science, Xi’an Shiyou University,
 Xi’an 710065, P. R. China}

\email{15829207515@163.com}

\author{Linan Zhong*}

\address{Department of Mathematics
Yanbian University
Yanji 133000, P. R. China}

\email{lnzhong@ybu.edu.cn}



\thanks{Supported by National Natural Science Foundation of China ( Nos.12261091 and 12526631).}

\thanks{*Corresponding author:Linan Zhong}

\begin{document}

\begin{abstract}
Under specific conditions, we investigate the rank variation in double cohomology following the gluing of an 
$n$-simplex to a given simplicial complex. As an explicit application of the established results, we provide a positive resolution to an open problem proposed in \cite{LPSS-2023}, which states the following: For any even integer $r$, there always exists a simplicial complex such that the rank of the double cohomology of its associated moment-angle complex equal to $r$.    

\end{abstract}

\maketitle

\section{Introduction}
Let $\K$ be a simplicial complex with vertex set $[m]=\{1,2,\cdots,m\}$, and let $\ZZ_{\K}$ denote its associated moment-angle complex. The cohomology ring $H^*(\ZZ_{\K})$ admits a decomposition as the direct sum $\bigoplus_{I\subset [m]} \widetilde{H}^*(\K_I)$, where each summand corresponds to the reduced simplicial cohomology of a full subcomplex 
$\K_I$ of $\K$. In \cite{LPSS-2023}, the authors establish a cochain complex structure $C\!H^*(\ZZ_{\K})$ on $H^*(\ZZ_{\K})$ and define the double cohomology $H\!H^*(\ZZ_{\K})$ as the cohomology of this cochain complex.

The definition of $H\!H^*(\ZZ_{\K})$ is motivated by the theory of persistent cohomology $P\!H^*(-)$. For a finite pseudometric space $S$, $P\!H^*(-)$ is defined via a parametrized family of cohomology groups that depend on a non-negative real number $t$ and a filtered simplicial complex $\{\K(t)\}_{t\ge 0}$ (see \cite{BLPSS-2024}). A fundamental property of $P\!H^*(-)$ is its stability: small perturbations of $S$ do not significantly alter $P\!H^*(S)$. However, the family of cohomology groups $\{H^*(\ZZ_{\K(t)})\}_{t\ge 0}$ fails to satisfy this stability. It is shown in \cite{BLPSS-2024} that the family of double cohomology groups $\{H\!H^*(\ZZ_{\K(t)})\}_{t\ge 0}$ does satisfy this stability. Consequently, double cohomology theory provides a novel framework for topological data analysis, leveraging the bigraded persistent cohomology of the corresponding moment-angle complexes. Moreover, since the double cohomology modules are smaller than the cohomology modules, they are more computationally accessible.

From a theoretical perspective, the authors of \cite{CCC-2024} establish an isomorphism between the $0$-degree \"uberhomology and the double homology of finite simplicial complexes.

In \cite{LPSS-2023}, the authors further establish several key properties of double cohomology and present computations for numerous examples. In particular, they show that the rank of $H\!H^*(\ZZ_{\K})$ is either $1$ (when $\K$ is a simplex) or an even integer (when $\K$ is not a simplex). Notably, all known examples satisfy $\opn{rank}\, H\!H^*(\ZZ_{\K})=2^n$ for some $n\ge 0$. This observation leads the authors to pose the following problem \cite[Question 8.9]{LPSS-2023}: Given a positive even integer $r$ that is not a power of two, does there exist a simplicial complex $\K$ for which $\opn{rank}\, H\!H^*(\ZZ_{\K})=r$?

Recently, \cite{Han-2023} presents an example in which $\opn{rank}\, H\!H^*(\ZZ_{\K})=6$, solving this problem for $r=6$. Furthermore, via the join operation, examples can be constructed for any $r$ of the form $r=2^{s+t}3^t$ with $s,t\ge 0$. For arbitrary even $r$, this question is fully resolved by the following theorem.

\begin{thm}
\label{thm: complex with any even rank}
There exists a family of simplicial complexes $\{\K_{2r} \mid r \ge 1\}$ for which $\opn{rank}\,H\!H^*(\ZZ_{\K_{2r}}) = 2r$.
\end{thm}

Our strategy for proving Theorem \ref{thm: complex with any even rank} is as follows. Let $\K'=\K\cup_{\sigma} \Delta^n$ denote the simplicial complex obtained by gluing an $n$-simplex to a nonempty simplicial complex $\K$ along a subcomplex $\sigma\in \K$. When $\sigma$ is
a proper (and possibly empty) face, it is known from \cite[Theorem 6.7]{LPSS-2023} that $\opn{rank}\, H\!H^*(\ZZ_{\K'})=1$ or $2$. Prior to our work, no explicit relationship had been established between $\opn{rank}\, H\!H^*(\ZZ_{\K'})$ and $\opn{rank}\, H\!H^*(\ZZ_{\K})$. 
However, when gluing an $n$-simplex along its boundary to a specific class of simplicial complexes, an explicit relationship between these two ranks emerges. In this paper, we develop a novel method to characterize this relationship, which further extends to the analysis of the connection between $H\!H^*(\ZZ_{\K'})$ and $H\!H^*(\ZZ_{\K})$ for arbitrary field coefficients. Specifically, we obtain the following results.

\begin{thm}
\label{thm: change of rank}
Let $\Delta^n$ denote the $n$-simplex with vertex set $[n+1]$, and let $\K$ be a simplicial complex with vertex set $[m]$, where $m \ge n+2 \ge 3$. Suppose that the following conditions hold:
\begin{enumerate}
\item $[n+1] \notin \K$;
\item For every index $j$ satisfying $n+2 \le j \le m$, and for every $i \in [n+1]$, the set $J = \{1, \cdots, n+1, j\}$ satisfies $J \setminus \{i\} \in \K$;
\item For every pair of indices $j, k$ with $n+2 \le j < k \le m$, let $J = \{1, \cdots, n+1, j, k\}$. Then $\opn{rank}\,\widetilde{H}^n(\K_J) \le 1$;
\item For every pair of indices $j, k$ with $n+2 \le j < k \le m$, let $J = \{1, \cdots, n+1, j, k\}$. Then either $J \setminus \{i\} \in \K$ for all 
$i \in [n+1]$, or $J \setminus \{i\} \notin \K$ for all $i \in [n+1]$.
\end{enumerate}
Let $\lambda_{[n+1]}\K := \K \cup_{\partial\Delta^n} \Delta^n$ denote the simplicial complex obtained by gluing an $n$-simplex (with vertex set $[n+1]$) to $\K$ along the boundary $\partial\Delta^n$. Then:
\begin{enumerate}
\item If there exists a pair of indices $j, k$ (with $n+2 \le j < k \le m$) such that $\opn{rank}\,\widetilde{H}^n(\K_J) = 1$ for $J = \{1, \cdots, n+1, j, k\}$, then 
$\opn{rank}\,H\!H^*(\ZZ_{\lambda_{[n+1]}\K}) = \opn{rank}\,H\!H^*(\ZZ_{\K}) - 2$;
\item Otherwise, $\opn{rank}\,H\!H^*(\ZZ_{\lambda_{[n+1]}\K}) = \opn{rank}\,H\!H^*(\ZZ_{\K})$.
\end{enumerate}
\end{thm}

Using Theorem \ref{thm: change of rank}, we can construct, for any even integer $r$, a simplicial complex $\K$ satisfying $\opn{rank}\,H\!H^*(\ZZ_{\K}) = r$, thus proving Theorem \ref{thm: complex with any even rank}.

\subsection*{Organization of the paper}
In Section \ref{sec: preliminaries}, we recall the definition of double cohomology and establish several useful lemmas. In Section \ref{sec: The proof of Theorem 1}, we provide a complete proof of Theorem \ref{thm: change of rank}. In Section \ref{sec: the proof of thm 2}, we complete the proof of Theorem \ref{thm: complex with any even rank} using Theorem \ref{thm: change of rank} in case $n=1$.


\section{Preliminaries}\label{sec: preliminaries}
Let $\K$ be a simplicial complex with vertex set $[m] = \{1, \cdots, m\}$. A $1$-element simplex $\{i\} \in \K$ is referred to as a vertex of $\K$. Throughout this paper, we assume that $\varnothing \in \K$ and $\{i\} \in \K$ for all $i \in [m]$ (i.e., $\K$ is a "vertex-full" simplicial complex). For each simplex $I \in \K$, define the topological space:
$$(D^2,S^1)^I := \left\{ (z_1,\cdots,z_m) \in (D^2)^m \mid |z_j| = 1 \text{ if } j \notin I \right\} \subset (D^2)^m.$$

Note that if $I \subset J$ (as simplices of $\K$), then $(D^2,S^1)^I$ is a natural subspace of $(D^2,S^1)^J$, inducing an inclusion of subspaces.

\begin{defn}[\cite{BPbook-2015}, Chapter 4]
\label{def: moment angle complex}
The moment-angle complex associated to $\K$ is defined as the union of these subspaces:
$\ZZ_{\K} := \bigcup_{I \in \K} (D^2,S^1)^I \subset (D^2)^m.$
\end{defn}

The cohomology ring of $\ZZ_{\K}$ admits a well-known decomposition via full subcomplexes of $\K$, as stated below.

\begin{prop}[\cite{BBP-2004}, \cite{BT-2000}]
\label{prop: cohomology ring of Zk}
There is an isomorphism of bigraded commutative algebras:
$$H^*(\ZZ_{\K}) \cong \bigoplus_{I \subset [m]} \widetilde{H}^*(\K_I),$$
where $\K_I \subset \K$ denotes the full subcomplex of $\K$ restricted to $I$ (i.e., $\K_I = \{ \sigma \in \K \mid \sigma \subset I \}$), and $\widetilde{H}^*(\K_I)$ is the reduced simplicial cohomology of $\K_I$.
\end{prop}

\begin{rem}\label{rem: decomposition}
By convention, $\widetilde{H}^*(\K_{\varnothing}) = \widetilde{H}^{-1}(\K_{\varnothing}) = \Z$ (the reduced cohomology of the empty simplicial complex). This isomorphism is a direct consequence of Hochster’s theorem, which relates the cohomology of moment-angle complexes to the combinatorics of  $\K$.
\end{rem}

\subsection*{Double cohomology}

To define the double cohomology of $\ZZ_{\K}$, we first introduce a differential operator on $H^*(\ZZ_{\K})$. For any $I \subset [m]$ and $i \in I$, the inclusion of full subcomplexes $\K_{I \setminus \{i\}} \hookrightarrow \K_I$ induces a homomorphism of reduced cohomology groups:
$$\psi_{p;i,I} : \widetilde{H}^p(\K_I) \to \widetilde{H}^p(\K_{I \setminus \{i\}}).$$ 

For a fixed $I \subset [m]$, define the differential component $d_p$ (acting on the $p$-th reduced cohomology summand) as
\begin{equation}
\label{equ: def of dp}
d_p = (-1)^{p+1} \sum_{i \in I} \varepsilon(i,I) \cdot \psi_{p;i,I} : \widetilde{H}^p(\K_I) \to \bigoplus_{i \in I} \widetilde{H}^p(\K_{I \setminus {i}}),
\end{equation}
where $\varepsilon(i,I) = (-1)^{\#\{j \in I \mid j < i\}}$ is the standard sign factor for simplicial enumeration. Combining these components with the decomposition in Proposition \ref{prop: cohomology ring of Zk}, we define a global differential $d$ on $H^*(\ZZ_{\K})$. The key property of this differential is:

\begin{lem}[\cite{LPSS-2023}]
\label{lem: dd=0}
The differential $d$ satisfies $d^2 = 0$, so $(H^*(\ZZ_{\K}), d)$ forms a cochain complex.
\end{lem}

We denote this cochain complex by $C\!H^*(\ZZ_{\K}) := (H^*(\ZZ_{\K}), d)$, and explicitly write it as
$$C\!H^*(\ZZ_{\K}) = \bigoplus_{p \ge -1} \left( \bigoplus_{I \subset [m]} \widetilde{H}^p(\K_I), d_p \right).$$

\begin{defn}
\label{def:double cohomology}
The double cohomology of $\ZZ_{\K}$ is defined as the cohomology of the cochain complex $C\!H^*(\ZZ_{\K})$:
$H\!H^*(\ZZ_{\K}) = H\left( H^*(\ZZ_{\K}), d \right).$
\end{defn}

\begin{rem}
Double cohomology can also be characterized algebraically via the Koszul algebra of $\K$’s face ring. For a detailed algebraic construction, we refer to \cite[Theorem 4.4]{LPSS-2023}.
\end{rem}

A fundamental property of the double cohomology rank is given by:

\begin{prop}[\cite{LPSS-2023}, Corollary 4.6]
\label{prop: euler char is 0}
If $\K$ is not a simplex, then the Euler characteristic of $H\!H^*(\ZZ_{\K})$ is zero. In particular, $\opn{rank}\, H\!H^*(\ZZ_{\K})$ (as a $\Z$-module or vector space over a field) is even.
\end{prop}

For the simplicial join of two simplicial complexes $\K$ and $\LL$ (denoted $\K*\LL$), the double cohomology behaves multiplicatively under mild conditions:

\begin{prop}[\cite{LPSS-2023}, Theorem 6.3]
\label{prop: simplicial join}
If either $H^*(\ZZ_{\K})$ or $H^*(\ZZ_{\LL})$ is projective over the coefficient ring, then there is an isomorphism of cochain complexes:
$$C\!H^*(\ZZ_{\K*\LL}) \cong C\!H^*(\ZZ_{\K}) \otimes C\!H^*(\ZZ_{\LL}).$$

For field coefficients (a case we frequently use later), this implies
\begin{equation*}
\begin{split}
H\!H^*(\ZZ_{\K*\LL}) &\cong H\!H^*(\ZZ_{\K}) \otimes H\!H^*(\ZZ_{\LL}),\\
\opn{rank}\, H\!H^*(\ZZ_{\K*\LL}) &= \opn{rank}\, H\!H^*(\ZZ_{\K}) \cdot \opn{rank}\, H\!H^*(\ZZ_{\LL}).
\end{split}
\end{equation*}
\end{prop}

In contrast, the double cohomology of wedge-decomposable simplicial complexes (simplicial complexes that split as a wedge of smaller complexes) has a fixed rank:\begin{prop}[\cite{Ruiz-2023}]
\label{prop: double coho of wedge}
For any wedge-decomposable simplicial complex $\K$, $H\!H^*(\ZZ_{\K})$ $\cong \Z \oplus \Z$ (so $\opn{rank}\, H\!H^*(\ZZ_{\K}) = 2$).
\end{prop}

We conclude this section with a technical lemma, which will be critical for decomposing cochain complexes and computing ranks in Theorem \ref{thm: change of rank}.
\begin{lem}
\label{lem: dim ker}
Let $\KK$ be a field (e.g. $\R$ or $\C$). Consider three cochain complexes over $\KK$: $(A_n, d'_n)_{n \ge 0}$, $(B_n, d''_n)_{n \ge 0}$, and $(C_n, d_n)_{n \ge 0}$. Suppose that the following holds for all $n \ge 0$:
\begin{enumerate}
\item $\opn{rank}\, A_n$ and $\opn{rank}\, B_n$ are finite;
\item $(B_n, d''_n)_{n \ge 0}$ is exact for $n\ge 1$, i.e., $\ker d''_n = \opn{Im}\, d''_{n+1}$ for all $n\ge 0$;
\item $C_n = A_n \oplus B_n$ as $\KK$-vector spaces;
\item The differential $d_n$ satisfies $d_n|_{B_{n+1}} = d''_n$, and $\pi_{1,n} \circ d_n|_{A_{n+1}} = d'_n$, where $\pi_{1,n} : A_n \oplus B_n \to A_n$ is the canonical projection.
\end{enumerate}
Then for all $k \ge 1$, we have:
\begin{enumerate}
\item $\opn{rank}\, \ker d_k = \opn{rank}\, \ker d'_k + \opn{rank}\, \ker d''_k$;
\item $\opn{rank}\, \opn{Im}\, d_k = \opn{rank}\, \opn{Im}\, d'_k + \opn{rank}\, \opn{Im}\, d''_k$.
\end{enumerate}
\end{lem}\begin{proof}
For $k \ge 1$, consider the segment of the complex $(C_n, d_n)$:
$$\cdots \to A_{k+2} \oplus B_{k+2} \xrightarrow{d_{k+1}} A_{k+1} \oplus B_{k+1} \xrightarrow{d_k} A_k \oplus B_k \xrightarrow{d_{k-1}} A_{k-1} \oplus B_{k-1} \to \cdots.$$

(1) Let $\opn{rank}\, \ker d'_k = m$, and let $\{x_1, \dots, x_m\} \subset A_{k+1}$ be a basis for $\ker d'_k$. For each $x_i$, Condition (4) implies $d_k(x_i) \in B_k$ (since $\pi_{1,k}(d_k(x_i)) = d'_k(x_i) = 0$). And $d_{k-1}|_{B_{k}} = d''_{k-1}$ implies 
$$d''_{k-1}(d_k(x_i)) = d_{k-1}(d_k(x_i)) = 0.$$
By Condition (2), there exists $y_i \in B_{k+1}$ such that 
$$d_k(x_i) = d''_k(y_i) = d_k(y_i).$$

Now, take any $x + y \in \ker d_k$ with $x \in A_{k+1}$, $y \in B_{k+1}$. Then $\pi_{1,k}(d_k(x + y)) = d'_k(x) = 0$, so $x = \sum_{i=1}^m c_i x_i$ for some $c_i \in \KK$. Substituting gives
$$d_k\left( \sum_{i=1}^m c_i x_i + y \right) = \sum_{i=1}^m c_i d_k(y_i) + d_k(y) = d''_k\left( \sum_{i=1}^m c_i y_i + y \right) = 0.$$

Then $$x+y=\sum_{i=1}^m c_i (x_i-y_i)+\sum_{i=1}^m c_i y_i+y$$
implies
\begin{equation*}
    \begin{split}
        \ker d_k\subset &\  \KK\{x_i-y_i:1\le i\le m\}+\ker d''_k\\
        = &\  \KK\{x_i-y_i:1\le i\le m\}\oplus\ker d''_k.
    \end{split}
\end{equation*}

Since $d_k(x_i - y_i) = 0$ and 
$d_k(\ker d''_k) = 0$, the reverse inclusion is trivial. Hence,
$$\opn{rank}\, \ker d_k = m + \opn{rank}\, \ker d''_k = \opn{rank}\, \ker d'_k + \opn{rank}\, \ker d''_k.$$

(2) Let $\opn{rank}\, \opn{Im}\, d'_k = s$, and let $\{z_1, \dots, z_s\} \subset A_k$ be a basis for $\opn{Im}\, d'_k$, with $z_i = d'_k(z'_i)$ for some $z'_i \in A_{k+1}$. For any $x \in A_{k+1}$ and $y\in B_{k+1}$, write $d'_k(x) = \sum_{i=1}^s b_i z_i$  with $b_i \in \KK$. Then $d'_k\left( x - \sum_{i=1}^s b_i z'_i \right) = 0$, so 
$$d_k\left( x - \sum_{i=1}^s b_i z'_i \right) \in B_k.$$ 
Then $d_{k-1}|_{B_{k}} = d''_{k-1}$ implies
$$d''_{k-1}d_k(x-\sum_{i=1}^s b_iz'_i)
=d_{k-1}d_k(x-\sum_{i=1}^s b_iz'_i)=0.$$
By Condition (2), there exists $w \in B_{k+1}$ such that 
$$d_k(x-\sum_{i=1}^s b_iz'_i)=d''_kw=d_kw.$$

Substituting, we get
$$d_k(x + y) = d_k\left( \sum_{i=1}^s b_i z'_i \right) + d_k(w + y) = \sum_{i=1}^s b_i \left( z_i + \pi_{2,k}(d_k(z'_i)) \right) + d''_k(w + y),$$
where $\pi_{2,k} : A_k \oplus B_k \to B_k$ is the canonical projection. This shows
\begin{equation*}
    \begin{split}
        \opn{Im} d_k = & \ \KK\{z_i+\pi_{2,k}d_kz'_i:1\le i\le s\}+
        \opn{Im} d''_k\\
        = & \ \KK\{z_i+\pi_{2,k}d_kz'_i:1\le i\le s\}\oplus
        \opn{Im} d''_k.
    \end{split}
\end{equation*}

Hence,
$\opn{rank}\, \opn{Im}\, d_k = s + \opn{rank}\, \opn{Im}\, d''_k = \opn{rank}\, \opn{Im}\, d'_k + \opn{rank}\, \opn{Im}\, d''_k.$
\end{proof}\begin{rem}
The finiteness condition (1) is not essential and can be removed (the proof extends to infinite ranks by linear algebra over fields). We include it here only because our applications involve finite-rank cohomology groups.
\end{rem}

\section{The proof of Theorem \ref{thm: change of rank}}
\label{sec: The proof of Theorem 1}
To simplify the rank computations, we fix the coefficient field as $\R$ for the entire section. We denote $\widetilde{H}^*(-) = \widetilde{H}^*(-; \R)$ and abbreviate $\lambda_{[n+1]}\K$ (the simplicial complex obtained by gluing an $n$-simplex $\Delta^n$ to $\K$ along $\partial\Delta^n$) as $\lambda\K$. We also assume that all conditions of Theorem \ref{thm: change of rank} hold throughout.

\subsection*{Compare cohomology of full subcomplexes}
First, we characterize how the full subcomplexes of $\lambda\K$ relate to those of $\K$. This is critical for comparing their cohomology groups.

\begin{lem}
\label{lem: lamda Kj and Kj}
For any subset $J \subset [m]$, the full subcomplex $(\lambda\K)_J$ (of $\lambda\K$ restricted to $J$) satisfies the following:
\begin{enumerate}
\item If $J$ does not contain $[n+1]$ (i.e., there exists $i \in [n+1]$ with $i \notin J$), then $(\lambda\K)_J = \K_J$;
\item If $J \supset [n+1]$, then $(\lambda\K)_J = \lambda(\K_J)$, where $\lambda(\K_J)$ denotes the simplicial complex obtained by gluing $\Delta^n$ to $\K_J$ along $\partial\Delta^n$.
\end{enumerate}
\end{lem}

\begin{proof}
    It is immediate since gluing $\Delta^n$ along $\partial\Delta^n \subset \K$ does not affect subcomplexes that do not contain $[n+1]$.
\end{proof}

Using Lemma \ref{lem: lamda Kj and Kj}, we now relate the reduced cohomology of $(\lambda\K)_J$ and $\K_J$. This step leverages the topological meaning of gluing a simplex along its boundary.
\begin{lem}
\label{lem: Hlamda Kj and HKj}
For any subset $J \subset [m]$, the reduced cohomology groups of $(\lambda\K)_J$ and $\K_J$ satisfy the following:
\begin{enumerate}
\item For all $s \notin \{n-1, n\}$, $\widetilde{H}^s((\lambda\K)_J) \cong \widetilde{H}^s(\K_J)$;
\item If $J \not\supset [n+1]$, then $\widetilde{H}^*((\lambda\K)_J) \cong \widetilde{H}^*(\K_J)$;
\item If $J = [n+1]$, then $\K_J = \partial\Delta^n$ and $(\lambda\K)_J = \Delta^n$ (by Condition (1) of Theorem \ref{thm: change of rank}). 
Thus,
$\widetilde{H}^n((\lambda\K)_{[n+1]}) = \widetilde{H}^n(\K_{[n+1]}) = \widetilde{H}^{n-1}((\lambda\K)_{[n+1]}) = 0, \quad \widetilde{H}^{n-1}(\K_{[n+1]}) = \R;$
\item If $J \supset [n+1]$ and $J \neq [n+1]$, then $(\lambda\K)_J \simeq \K_J \vee S^n$ and
$$\widetilde{H}^n((\lambda\K)_J) \cong \widetilde{H}^n(\K_J) \oplus \R, \quad \widetilde{H}^{n-1}((\lambda\K)_J) \cong \widetilde{H}^{n-1}(\K_J).$$
\end{enumerate}
\end{lem}
\begin{proof}
   (1) comes from the fact that gluing $\Delta^n$ only affects the cohomology in dimensions $n-1$ and $n$. (2) and (3) follow directly from Lemma \ref{lem: lamda Kj and Kj}.

   (4) By Condition (2) of Theorem \ref{thm: change of rank}, $\partial\Delta^n$ is contractible in $\K_J$. Then gluing $\Delta^n$ to $\K_J$ gives $(\lambda\K)_J \simeq \K_J \vee S^n$. The results of their reduced cohomology are immediate.

\end{proof}

\subsection*{Rank change of double cohomology in dimensions not equal to $n$}

By Lemma \ref{lem: Hlamda Kj and HKj} (1), the double cohomology of $\ZZ_{\lambda\K}$ and $\ZZ_{\K}$ agree in all dimensions except $n-1$ and $n$. This simplifies our task to analyze only these two critical dimensions.
\begin{cor}
\label{cor: rank H*}
For all $s \notin \{n-1, n\}$, $\opn{rank}\, H\!H^s(\ZZ_{\lambda\K}) = \opn{rank}\, H\!H^s(\ZZ_{\K})$.
\end{cor}

We first compute the rank of $H\!H^{n-1}(\ZZ_{\lambda\K})$ by comparing the cochain complexes of $\ZZ_{\lambda\K}$ and $\ZZ_{\K}$ in dimension $n-1$.

\begin{prop}
\label{prop: rank H n-1}
$\opn{rank}\, H\!H^{n-1}(\ZZ_{\lambda\K}) = \opn{rank}\, H\!H^{n-1}(\ZZ_{\K}) - 1$.
\end{prop}\begin{proof}
By Lemma \ref{lem: Hlamda Kj and HKj}, for $J \neq [n+1]$, 
$\widetilde{H}^{n-1}((\lambda\K)_J) \cong \widetilde{H}^{n-1}(\K_J)$. For $J = [n+1]$, $$\widetilde{H}^{n-1}((\lambda\K)_{[n+1]}) = 0,\quad \widetilde{H}^{n-1}(\K_{[n+1]}) = \R_{[n+1]}.$$
Here we use the notation $\R_{[n+1]}\cong\R$ to denote the cohomology of the component $\K_{[n+1]}$. 

The cochain complex $C\!H^{n-1}(\ZZ_{\K})$ (the component of $C\!H^*(\ZZ_{\K})$ in dimension $n-1$) is
$$\left( \bigoplus_{J \subset [m]} \widetilde{H}^{n-1}(\K_J), d_{\K} \right) = \left( \bigoplus_{J \subset [m]} \widetilde{H}^{n-1}((\lambda\K)_J) \oplus \R_{[n+1]}, d_{\K} \right).$$

We now show that $\R_{[n+1]}$ is a direct summand of the cochain complex that does not interact with other components. 

For $J$ with $|J| = n+2$ and $J \supset [n+1]$, Condition (1) and (2) of Theorem \ref{thm: change of rank} imply that $\K_J$ is contractible, so $\widetilde{H}^{n-1}(\K_J) = 0$. 

For $J \subset [n+1]$ with $|J| = n$, Condition (2) of Theorem \ref{thm: change of rank} implies $\K_J$ is an $(n-1)$-simplex, so
$\widetilde{H}^{n-1}(\K_J) = 0$. Thus, all possible homomorphisms involve $\R_{[n+1]}$
$$\bigoplus_{|J|=n+2} \widetilde{H}^{n-1}(\K_J) \to \R_{[n+1]} \to \bigoplus_{|J|=n} \widetilde{H}^{n-1}(\K_J)$$
are both trivial. This means $\R_{[n+1]}$ is an "isolated" summand in the cochain complex, which exactly contributes $1$ to the rank of $H\!H^{n-1}(\ZZ_{\K})$. So,
$$\opn{rank}\, H\!H^{n-1}(\ZZ_{\lambda\K}) = \opn{rank}\, H\!H^{n-1}(\ZZ_{\K}) - 1.$$
\end{proof}

\subsection*{Rank change of double cohomology in dimension $n$}
To complete the proof of Theorem \ref{thm: change of rank}, we need to analyze the rank of $H\!H^n(\ZZ_{\lambda\K})$ and $H\!H^n(\ZZ_{\K})$. This requires decomposing the cochain complex $C\!H^n(\ZZ_{\K})$ using Lemma \ref{lem: dim ker} and the following lemma.

\begin{lem}
    \label{lem: alpha element}
    There exists a nonzero element $\alpha_{[m]}\in \widetilde{H}^n((\lambda\K)_{[m]})=\widetilde{H}^n(\lambda\K)$ such that $\alpha_{[m]}$ is in the kernel of $\widetilde{H}^n((\lambda \K)_{[m]})\to \widetilde{H}^n(\K_{[m]})$, which is induced by inclusion. For any $J\subset [m]$, let $\alpha_J$ be the image of $\alpha_{[m]}$ in the homomorphism  $$\widetilde{H}^n((\lambda\K)_{[m]})\to \widetilde{H}^n((\lambda\K)_J)$$ induced by inclusion. Then  
    \begin{enumerate}
        \item For $J \supset [n+1]$ and $J\not =[n+1]$, $\alpha_J\in \widetilde{H}^n((\lambda\K)_J)$ is nonzero. And the kernel of the homomorphism $\widetilde{H}^n((\lambda\K)_J)\to \widetilde{H}^n(\K_J)$ induced by inclusion is $\R\{\alpha_J\}$.
        \item For other $J\subset [m]$, $\alpha_J=0$.
    \end{enumerate}
\end{lem}
\begin{proof}
    By Lemma \ref{lem: Hlamda Kj and HKj} (3) (4) and the duality between homology and cohomology,
\begin{equation}
\label{eq: cohomology of lamda K, K}
    \begin{split}
        \widetilde{H}^n((\lambda\K)_{J})\cong
        \begin{cases}
        \widetilde{H}^n(\K_J)\oplus\R, & \text{ if }
        [n+1]\subset J \text{ and } J\not =[n+1],\\
        \widetilde{H}^n(\K_J), & \text{ otherwise}.
        \end{cases}\\
        \widetilde{H}_n((\lambda\K)_{J})\cong
        \begin{cases}
        \widetilde{H}_n(\K_J)\oplus\R, & \text{ if }
        [n+1]\subset J \text{ and } J\not =[n+1],\\
        \widetilde{H}_n(\K_J), & \text{ otherwise}.
        \end{cases}
    \end{split}
\end{equation}

Firstly, we construct a nonzero element $\alpha_{[m]}\in \widetilde{H}^n((\lambda\K)_{[m]})=\widetilde{H}^n(\lambda\K)$ such that
$\alpha_{[m]}$ is in the kernel of $\widetilde{H}^n((\lambda
\K)_{[m]})\to \widetilde{H}^n(\K_{[m]})$,
which is induced by inclusion.  

Let $C_i(-), Z_i(-), B_i(-)$ denote the spaces of $i$-chains, $i$-cycles, and $i$-boundaries, respectively.

By Condition (1) of Theorem \ref{thm: change of rank}, the $n$-chain $\{1,\cdots,n+1\}\not\in C_n(\K_{[m]})$. Let $\partial\{1,\cdots,n+1,n+2\}$ be the boundary of the $(n+1)$-chain $\{1,\cdots,n+1,n+2\}$. Then, $\partial\{1,\cdots,n+1,n+2\}\in Z_n((\lambda\K)_{[m]})$ by Condition (2) of Theorem \ref{thm: change of rank}.

Note that $(\lambda\K)_{[m]}=\K_{[m]}\cup_{\partial\Delta^n} \Delta^n$, where $\Delta^n$ denotes the $n$-simplex with vertex set $[n+1]$. So $C_{n+1}((\lambda\K)_{[m]})=C_{n+1}(\K_{[m]})$ and $B_n((\lambda\K)_{[m]})=B_n(\K_{[m]})$. Thus $\{1,\cdots,n+1\}\not\in B_n((\lambda\K)_{[m]})$, which implies the $n$-cycle $\partial\{1,\cdots,n+1,n+2\}\not\in B_n((\lambda\K)_{[m]})$. Combined with \eqref{eq: cohomology of lamda K, K}, we have
\begin{equation}
    \label{eq: Z n(lamda K)}
    Z_n((\lambda\K)_{[m]})=Z_n(\K_{[m]})+\R\{\partial\{1,\cdots,n+1,n+2\}\}.
\end{equation}
Then the isomorphism $H^*(-;\R)\cong \opn{Hom}_{\R} (H_*(-;\R),\R)$ implies
\begin{equation*}
    \begin{split}
        H^n((\lambda\K)_{[m]})=H^n(\K_{[m]})\oplus\R\{[\partial\{1,\cdots,n+1,n+2\}]^*\},\\
    \end{split}
\end{equation*}
where $[\partial\{1,\cdots,n+1,n+2\}]^*$ denotes the cohomologous class represented by the dual of the homologous class $[\partial\{1,\cdots,n+1,n+2\}]$. Thus we construct a nonzero element $[\partial\{1,\cdots,n+1,n+2\}]^*=:\alpha_{[m]}\in \widetilde{H}^n((\lambda\K)_{[m]})=\widetilde{H}^n(\lambda\K)$ such that
the kernel of $\widetilde{H}^n((\lambda
\K)_{[m]})\to \widetilde{H}^n(\K_{[m]})$ induced by inclusion is $\R\{\alpha_{[m]}\}$.  

For any $J\subset [m]$, the commutative diagram
$$\xymatrix{
& \widetilde{H}^n((\lambda\K)_{[m]})\ar@{->}[d] \ar@{->}[r]
& \widetilde{H}^n((\lambda\K)_J)\ar@{->}[d]\\
& \widetilde{H}^n(\K_{[m]})\ar@{->}[r]
& \widetilde{H}^n(\K_J)\\
}$$
induced by inclusions show that $\alpha_J\in \ker (\widetilde{H}^n((\lambda\K)_J)\to \widetilde{H}^n(\K_J))$.

For $J \supset [n+1]$ and $J\not =[n+1]$, Condition (2) of Theorem \ref{thm: change of rank} shows that there exists
some $n+2\le i\le m$ such that the cycle $\partial\{1,\cdots,n+1,i\}\in Z_n((\lambda\K)_J)$. By \eqref{eq: Z n(lamda K)}, we can write
$$\partial\{1,\cdots,n+1,i\}=r\partial\{1,\cdots,n+1,n+2\}+z,$$
for some $r\neq 0$ and $z\in Z_n(\K_{[m]})$. Since $\partial\{1,\cdots,n+1,i\}$ contains the $n$-chain $\{1,\cdots,n+1\}\notin Z_n(\K_{[m]})$. Thus, $r=1$ and $\alpha_{[m]}([\partial\{1,\cdots,n+1,i\}])\neq 0$, which implies $\alpha_J\neq 0$. By \eqref{eq: cohomology of lamda K, K}, 
$$\ker(\widetilde{H}^n((\lambda\K)_J)\to \widetilde{H}^n(\K_J))=\R\{\alpha_J\}.$$

Moreover, the inclusion induces an isomorphism
\begin{equation}\label{equ: isomorphism}
    \widetilde{H}^n((\lambda\K)_J)/\R\{\alpha_J\}\cong
\widetilde{H}^n(\K_J).
\end{equation}

For other $J\subset [m]$, by \eqref{eq: cohomology of lamda K, K}, $\alpha_J\in \ker (\widetilde{H}^n((\lambda\K)_J)\to \widetilde{H}^n(\K_J))=0$.
\end{proof}

Next, we analyze the rank change of double cohomology in dimension $n$
by decomposing the cochain complex $C\!H^n(\ZZ_{\K})$.

\begin{prop}
\label{prop: rank H n}
\begin{enumerate}
\item If there exists $J_0 = \{1, \dots, n+1, j, k\}$  with $n+2 \le j < k \le m$  such that $\opn{rank}\, \widetilde{H}^n(\K_{J_0}) = 1$, then $\opn{rank}\, H\!H^n(\ZZ_{\lambda\K}) = \opn{rank}\, H\!H^n(\ZZ_{\K})$ $- 1$;
\item Otherwise (i.e. $\widetilde{H}^n(\K_J) = 0$ for all $J \supset [n+1]$ with $|J| = n+3$), $\opn{rank}\, H\!H^n(\ZZ_{\lambda\K}) = \opn{rank}\, H\!H^n(\ZZ_{\K}) + 1$.
\end{enumerate}
\end{prop}

\begin{proof}
Let 
$$A_k=\bigoplus_{J\subset [m]:\ |J|=k} \widetilde{H}^n((\lambda\K)_J)/\R\{\alpha_J\},\ 
B_k=\bigoplus_{J\subset [m]:\ |J|=k} \R\{\alpha_J\}.$$
Then the cochain complex $(\bigoplus_{J\subset [m]} \widetilde{H}^n((\lambda\K)_J),d_{\lambda\K})$ satisfies
\begin{enumerate}
    \item 
    $\bigoplus_{J\subset [m]:\ |J|=k} \widetilde{H}^n((\lambda\K)_J)
    = A_k\oplus B_k$.
    \item 
    $(A_k,\pi_{1,k}d_{\lambda\K,{k-1}}|_{A_k})\cong (\bigoplus_{J\subset [m]:\ |J|=k} \widetilde{H}^n(\K_J),d_{\K,{k-1}})$, which is induced by the isomorphism \eqref{equ: isomorphism}. 
    Here, $\pi_{1,k}$ denotes the projective map: $A_k\oplus B_k\to A_k$.
\end{enumerate}

For $J\subset [m]$ with $[n+1]\subset J$, $J\not =[n+1]$, we can write
$$J=\{1,\cdots,n+1,i_1,\cdots,i_j\},$$
with some $j\ge 1$ and $n+2\le i_1< \cdots < i_j\le m$.
Then \eqref{equ: def of dp} shows 
\begin{equation}
\label{eq: d alpha}
d_{\lambda\K}(\alpha_J)=
\sum_{k=1}^j (-1)^{k-1} \alpha_{J\setminus \{i_k\}}.    
\end{equation}

Consider the simplex $\Delta^{m-n-2}$ on $[m-n-1]$.
For $j\ge 0$, the space $C_{j+1}(\Delta^{m-n-2})$ is spanned by
all $(j+1)$-chains as $\{i_1,\cdots,i_{j+2}\}$ with $1\le i_1<\cdots <i_{j+2}\le m-n-1$.
The boundary operator 
$\partial_j:C_{j+1}(\Delta^{m-n-2})\to C_{j}(\Delta^{m-n-2})$
is defined by
$$\partial_j \{i_1,\cdots,i_{j+2}\}
=\sum_{k=1}^{j+2} (-1)^{k-1} \{i_1,\cdots,i_{j+2}\}
\!\setminus\! \{i_k\}.$$
The fact $H_*(\Delta^{m-n-2})=H_0(\Delta^{m-n-2})=\R$ implies 
$\ker \partial_j=\opn{Im} \partial_{j+1}$ for $j\ge 0$
and $C_0(\Delta^{m-n-2})/\opn{Im} \partial_0=\R$.

By comparison of the two complexes mentioned above, we have
\begin{equation}
    \label{eq: isomorphism complexes}
    (B_k,d_{\lambda\K,k-1}|_{B_k})\cong (C_{k-n-2}(\Delta^{m-n-2}),\partial_{k-n-3}), \text{ for } k\ge n+2.
\end{equation}
Moreover, the latter is exact for $k\ge n+3$.
Thus, by Lemma \ref{lem: dim ker}, for $k\ge n+3$,
\begin{equation}
\label{eq: rank ker Im for k>=n+3}
    \begin{split}
        \opn{rank}\ker d_{\lambda\K,k}= &\ 
        \opn{rank}\ker d_{\K,k}+
        \opn{rank}\ker \partial_{k-n-2},\\
        \opn{rank}\opn{Im} d_{\lambda\K,k}= &\ 
        \opn{rank}\opn{Im} d_{\K,k}+
        \opn{rank}\opn{Im} \partial_{k-n-2}.\\ 
    \end{split}
\end{equation}
Then
\begin{equation}
    \label{eq: k>=n+4}
    \opn{rank}(\ker d_{\lambda\K,k}/
    \opn{Im} d_{\lambda\K,k+1})
    = \opn{rank}(\ker d_{\K,k}/
    \opn{Im} d_{\K,k+1}), \text{ for } k\ge n+3.
\end{equation}
Note that $B_k=0$ when $k\le n+1$. Thus, for $k\le n$, \eqref{eq: k>=n+4}
is also valid.

Now, it suffices to compute the above rank for $k=n+1,n+2$. We can write 
$$\bigoplus_{J\subset [m]:\ |J|=n+3} \widetilde{H}^n(\K_J)=C\bigoplus D,$$
where
$$C=\bigoplus_{\substack{J\subset [m]:\ |J|=n+3\\ [n+1]\subset J}} \widetilde{H}^n(\K_J),\ 
D=\bigoplus_{\substack{J\subset [m]:\ |J|=n+3\\ [n+1]\not\subset J}} \widetilde{H}^n(\K_J).$$

Firstly, we claim $d_{\K,n+2}|_C=0$. It suffices to prove that $\widetilde{H}^n(\K_{J\setminus \{j\}})=0$ for any $j\in J\subset [m]$ with $|J|=n+3$ and $[n+1]\subset J$. In fact, according to Condition (4) of Theorem \ref{thm: change of rank}, either $J\!\setminus\! \{i\}\in \K$ for all $1\le i\le n+1$, or $J\!\setminus\! \{i\}\not\in \K$ for all $1\le i\le n+1$. 

If $J\!\setminus\! \{i\}\in \K$ for all $1\le i\le n+1$, then by Condition (2) of Theorem \ref{thm: change of rank}, $\K_{J\!\setminus\! \{j\}}$
is contractible for any $j\in J$ and the claim is trivial. 

Assume $J\!\setminus\! \{i\}\not\in \K$ for all $1\le i\le n+1$.
By Condition (2) of Theorem \ref{thm: change of rank}, for any $j\in J$, $\K_J$ can be obtained by two steps: (i) Glue all $n$-faces of $\K_J$ with the vertex $j$ along a $(n-1)$-subcomplex of $\K_{J\!\setminus\! \{j\}}$ that is contractible; (ii) then, glue other faces in dimensions less than $n+1$. Thus, we have 
\begin{equation}
    \label{eq: decomposition of K J}
    \K_J\simeq \K_{J\setminus \{j\}}\vee S^n\vee \{\text{other spheres that may not exist}\}.
\end{equation}
Then
$$\opn{rank} \widetilde{H}^n(\K_{J\setminus \{j\}})\le
\opn{rank} \widetilde{H}^n(\K_J)-1\le 0,$$
which shows $\widetilde{H}^n(\K_{J\setminus \{j\}})=0$.

This claim implies that
\begin{equation}
\label{eq: ker and Im 1}
    \begin{split}
        \ker d_{\K,n+2}= &\ C\bigoplus \ker d_{\K,n+2}|_D,\\
        \opn{Im} d_{\K,n+2}= &\ \opn{Im} d_{\K,n+2}|_D.
    \end{split}
\end{equation}

Recall that for $[n+1]\not\subset J\subset [m]$, $(\lambda\K)_J= \K_J$ and $\alpha_J=0$. So we can write
$$\bigoplus_{J\subset [m]:\ |J|=n+3}
\widetilde{H}^n((\lambda\K)_J)=C'\bigoplus D\bigoplus B_{n+3},$$
where
$$C'=\bigoplus_{\substack{J\subset [m]:\ |J|=n+3\\ [n+1]\subset J}} \widetilde{H}^n((\lambda\K)_J)/\R\{\alpha_J\}\cong C.$$
The claim above and the isomorphism
$$(A_k,\pi_{1,k}d_{\lambda\K,{k-1}}|_{A_k})\cong (\bigoplus_{J\subset [m]:\ |J|=k} \widetilde{H}^n(\K_J),d_{\K,{k-1}})$$ 
show $d_{\lambda\K,n+2}(C')\subset B_{n+2}$.
By the definition of $B_{n+2}$ and Lemma \ref{lem: alpha element}, we have $d_{\lambda
\K,n+2}(D)\subset A_{n+2}$. Recall that $d_{\lambda\K,n+2}(B_{n+3})\subset B_{n+2}$.

Hence, we have
\begin{equation}
\label{eq: ker and Im 2}
    \begin{split}
        \ker d_{\lambda\K,n+2} = &\ \ker d_{\lambda\K,n+2}|_D\bigoplus
        \ker d_{\lambda\K,n+2}|_{C'\oplus B_{n+3}}\\
        \cong &\ \ker d_{\K,n+2}|_D\bigoplus
        \ker d_{\lambda\K,n+2}|_{C'\oplus B_{n+3}},\\
        \opn{Im} d_{\lambda\K,n+2} = &\ \opn{Im} d_{\lambda\K,n+2}|_D\bigoplus
        \opn{Im} d_{\lambda\K,n+2}|_{C'\oplus B_{n+3}}\\
        \cong &\ \opn{Im} d_{\K,n+2}|_D\bigoplus
        \opn{Im} d_{\lambda\K,n+2}|_{C'\oplus B_{n+3}}.
    \end{split}
\end{equation}

Next, we divide into two cases.

(1) Assume that there exists some $J_0=\{1,\cdots,n+1,$ $j,k\}$ with $n+2\le j<k\le m$ such that $\opn{rank} \widetilde{H}^n(\K_{J_0})=1$.
Now we claim $d_{\lambda\K,n+2}(C'\oplus B_{n+3})=B_{n+2}$. 

By the isomorphism \eqref{eq: isomorphism complexes}, we have
$$B_{n+2}/d_{\lambda\K,n+2}(B_{n+3})
\cong C_0(\Delta^{m-n-2})/\opn{Im} \partial_0=\R.$$
So, it suffices to show that $d_{\lambda\K,n+2}(C'\oplus B_{n+3})\not =
d_{\lambda\K,n+2}(B_{n+3})$.

If $J_0\!\setminus\! \{i\}\in \K$ for all $1\le i\le n+1$, then $\K_{J_0}$ is contractible, which is contradictory to $\widetilde{H}^n(\K_{J_0})\neq 0$.
By Condition (4) of Theorem \ref{thm: change of rank}, $J_0\!\setminus\! \{i\}\not\in \K$ for all $1\le i\le n+1$. Thus, $C_{n+1}((\lambda\K)_{J_0})=C_{n+1}(\K_{J_0})=0$ and $B_n((\lambda\K)_{J_0})=B_n(\K_{J_0})=0$. Similarly to \eqref{eq: decomposition of K J}, we have
$$\K_{J_0}\simeq \K_{{J_0}\setminus \{i\}}\vee S^n\vee \{\text{other spheres that may not exist}\}, \text{ for any } i\in J_0.$$

Since $\widetilde{H}^n(\K_{J_0})=\R$, $\widetilde{H}^n(\K_{{J_0}\setminus \{j\}})=\widetilde{H}^n(\K_{{J_0}\setminus \{k\}})=0$ and 
$$\widetilde{H}^n((\lambda\K)_{{J_0}\setminus \{i\}})=\widetilde{H}^n(\K_{{J_0}\setminus \{i\}})=0, \text{ for } i\in [n+1].$$
Thus, $d_{\lambda\K,n+2}(\widetilde{H}^n((\lambda\K)_{J_0}))\subset \widetilde{H}^n((\lambda\K)_{J_0\setminus\{j\}})
\bigoplus \widetilde{H}^n((\lambda\K)_{J_0\setminus\{k\}})$.

By \eqref{eq: cohomology of lamda K, K} and Lemma \ref{lem: alpha element}, 
\begin{equation*}
    \begin{split}
        \widetilde{H}^n((\lambda\K)_{J_0})= & \
        \widetilde{H}^n(\K_{J_0})\bigoplus\R\{\alpha_{J_0}\}=\R\bigoplus\R\{\alpha_{J_0}\},\\
        \widetilde{H}^n((\lambda\K)_{J_0\setminus \{j\}})= & \
        \widetilde{H}^n(\K_{J_0\setminus \{j\}})\bigoplus\R\{\alpha_{J_0\setminus \{j\}}\}=\R\{\alpha_{J_0\setminus \{j\}}\},\\
        \widetilde{H}^n((\lambda\K)_{J_0\setminus \{k\}})= & \
        \widetilde{H}^n(\K_{J_0\setminus \{k\}})\bigoplus\R\{\alpha_{J_0\setminus \{k\}}\}=\R\{\alpha_{J_0\setminus \{k\}}\}.\\
    \end{split}
\end{equation*}

Recall that $(\lambda\K)_{J_0\setminus \{j\}}=\K_{J_0\setminus \{j\}}\cup_{\partial\Delta^n} \Delta^n$, where $\Delta^n$ denotes the $n$-simplex with vertex set $[n+1]$.
Since 
$$C_{n+1}((\lambda\K)_{J_0\setminus \{j\}})=C_{n+1}(\K_{J_0\setminus \{j\}})\subset C_{n+1}((\lambda\K)_{J_0})=0,$$
$B_n((\lambda\K)_{J_0\setminus \{j\}})=B_n(\K_{J_0\setminus \{j\}})=0$. Similarly to \eqref{eq: Z n(lamda K)}, we have
\begin{equation*}
    Z_n((\lambda\K)_{J_0\setminus \{j\}})=Z_n(\K_{J_0\setminus \{j\}})+\R\{\partial(J_0\setminus \{j\})\}.
\end{equation*}
Thus,
\begin{equation*}
    \begin{split}
        H^n((\lambda\K)_{J_0\setminus \{j\}})=& \ H^n(\K_{J_0\setminus \{j\}})\oplus\R\{[\partial(J_0\setminus \{j\})]^*\},
    \end{split}
\end{equation*} 
where the cohomologous class $[(\partial(J_0\setminus \{j\}))]^*$ denotes the dual of the homologous class $[(\partial(J_0\setminus \{j\}))]$. This implies $\widetilde{H}^n((\lambda\K)_{J_0\setminus \{j\}})=\R\{[(\partial(J_0\setminus \{j\}))]^*\}$. Similarly, $\widetilde{H}^n((\lambda\K)_{J_0\setminus \{k\}})=\R\{[(\partial(J_0\setminus \{k\}))]^*\}$ and $\partial(J_0\setminus \{j\}),\partial(J_0\setminus \{k\})\in Z_n((\lambda\K)_{J_0})\setminus B_n((\lambda\K)_{J_0})$.

Note that the cycle $\beta:=\partial(J_0\setminus \{j\})-\partial(J_0\setminus \{k\})\in Z_n(\K_{J_0})$ and $B_n((\lambda\K)_{J_0})=B_n(\K_{J_0})=0$. Thus, 
$$\widetilde{H}^n((\lambda\K)_{J_0})=\R[[\partial(J_0\setminus \{j\})]^*]\bigoplus \R\{[\beta]^*\}.$$
Since $[\beta]^*([\partial(J_0\setminus \{j\})])= 0$ and
$$[\beta]^*([\partial(J_0\setminus \{k\})])
=[\beta]^*([\partial(J_0\setminus \{j\})]-[\partial(J_0\setminus \{j\})-\partial(J_0\setminus \{k\})])=-1,$$
$d_{\lambda\K,n+2}([\beta]^*)\neq 0\in\widetilde{H}^n((\lambda\K)_{J_0\setminus \{k\}})\cong \R\{\alpha_{J_0\setminus\{k\}}\}$, which implies $d_{\lambda\K,n+2}([\beta]^*)=x\alpha_{J_0\setminus\{k\}}$ for some nonzero $x\in\R$. 

\eqref{eq: d alpha} shows that for any element in
$d_{\lambda\K,n+2}(B_{n+3})$, the sum of coefficients is $0$.
Thus, $d_{\lambda\K,n+2}([\beta]^*)\not\in d_{\lambda\K,n+2}(B_{n+3})$,
which implies the claim.

Combined with this claim and the isomorphisms \eqref{eq: isomorphism complexes}, \eqref{eq: ker and Im 1} and \eqref{eq: ker and Im 2}, we have 
\begin{equation*}
    \begin{split}
        & \ \opn{rank} \ker d_{\lambda\K,n+2}\\
        = & \  \opn{rank} \ker d_{\K,n+2}|_D
        + \opn{rank} \ker d_{\lambda\K,n+2}|_{C'\oplus B_{n+3}}\\
        = & \  \opn{rank} \ker d_{\K,n+2}|_D
        + \opn{rank} C' + \opn{rank} B_{n+3} - \opn{rank} B_{n+2}\\
        = & \  (\opn{rank} \ker d_{\K,n+2}|_D
        + \opn{rank} C) + \opn{rank} C_1(\Delta^{m-n-2})
        -\opn{rank} C_0(\Delta^{m-n-2})\\
        = & \  \opn{rank} \ker d_{\K,n+2}+
        (\opn{rank} \ker \partial_0 +\opn{rank} \opn{Im} \partial_0)
        -(\opn{rank} \opn{Im} \partial_0+1)\\
        = & \ \opn{rank} \ker d_{\K,n+2}+
        \opn{rank} \ker \partial_0-1.
    \end{split}
\end{equation*}
By \eqref{eq: rank ker Im for k>=n+3},
    $$\opn{rank} \opn{Im} d_{\lambda\K,n+3}
= \opn{rank} \opn{Im} d_{\K,n+3} + \opn{rank} \opn{Im} \partial_1.$$

Hence,
\begin{equation}
\label{eq: rank n+2 1 case}
    \begin{split}
         & \ \opn{rank} \, (\ker d_{\lambda\K,n+2}/\opn{Im} d_{\lambda\K,n+3})\\
        = & \ \opn{rank} \ker d_{\lambda\K,n+2}-
        \opn{rank} \opn{Im} d_{\lambda\K,n+3}\\
        = & \ \opn{rank} \, (\ker d_{\K,n+2}/\opn{Im} d_{\K,n+3})-1.
    \end{split}
\end{equation}

Note that $B_{n+1}=0$. Combined with the claim mentioned above and the isomorphisms \eqref{eq: isomorphism complexes}, \eqref{eq: ker and Im 1} and \eqref{eq: ker and Im 2}, we also have
\begin{equation}
\label{eq: rank n+1 1 case}
    \begin{split}
         & \ \opn{rank} \, (\ker d_{\lambda\K,n+1}/\opn{Im} d_{\lambda\K,n+2})\\
        = & \ \opn{rank} \ker d_{\lambda\K,n+1}-
        \opn{rank} \opn{Im} d_{\lambda\K,n+2}\\
        = & \ \opn{rank} \ker d_{\K,n+1}+\opn{rank} B_{n+2}-
        (\opn{rank} \opn{Im} d_{\K,n+2}|_D+
        \opn{rank} \opn{Im} d_{\lambda\K,n+2}|_{C'\oplus B_{n+3}})\\
        = & \ \opn{rank} \, (\ker d_{\K,n+1}/\opn{Im} d_{\K,n+2}).
    \end{split}
\end{equation}
By \eqref{eq: k>=n+4}, \eqref{eq: rank n+2 1 case} and \eqref{eq: rank n+1 1 case}, we have
$$\opn{rank} H\!H^n(\ZZ_{\lambda\K})=\opn{rank} H\!H^n
(\ZZ_{\K})-1.$$

(2) Assume $\widetilde{H}^n(\K_J)=0$ 
for any $J\subset [m]$ with $[n+1]\subset J,\ |J|=n+3$.
Then $C=C'=0$. By \eqref{eq: rank n+2 1 case} and \eqref{eq: rank n+1 1 case}, we have
\begin{equation*}
    \begin{split}
        & \ \opn{rank} \ker d_{\lambda\K,n+2}\\
        = & \  \opn{rank} \ker d_{\K,n+2}|_D
        + \opn{rank} \ker d_{\lambda\K,n+2}|_{B_{n+3}}\\
        = & \  \opn{rank} \ker d_{\K,n+2}
        + \opn{rank} \ker \partial_0.\\
    \end{split}
\end{equation*}
By \eqref{eq: rank ker Im for k>=n+3}, we have
\begin{equation}
    \label{eq: rank ker Im 2 case 1}
    \begin{split}
        & \ \opn{rank} \, (\ker d_{\lambda\K,n+2}/\opn{Im} d_{\lambda\K,n+3})\\
        = & \ \opn{rank} \ker d_{\K,n+2}
         + \opn{rank} \ker \partial_0
         - ( \opn{rank} \opn{Im} d_{\K,n+3} + \opn{rank} \opn{Im} \partial_1 )\\
         = & \ \opn{rank} \, (\ker d_{\K,n+2}/\opn{Im} d_{\K,n+3}).\\
    \end{split}
\end{equation}

Note that $B_{n+1}=0$. Combined with the isomorphisms \eqref{eq: isomorphism complexes}, \eqref{eq: ker and Im 1} and \eqref{eq: ker and Im 2}, we also have
\begin{equation}
    \label{eq: rank ker Im 2 case 2}
    \begin{split}
         & \ \opn{rank} \, (\ker d_{\lambda\K,n+1}/\opn{Im} d_{\lambda\K,n+2})\\
        = & \ \opn{rank} \ker d_{\lambda\K,n+1}-
        \opn{rank} \opn{Im} d_{\lambda\K,n+2}\\
        = & \ (\opn{rank} \ker d_{\K,n+1}+\opn{rank} B_{n+2})
        -
        (\opn{rank} \opn{Im} d_{\K,n+2}+
        \opn{rank} \opn{Im}
        d_{\lambda\K,n+2}|_{B_{n+3}})\\
        = & \ \opn{rank} \ker d_{\K,n+1}+\opn{rank} C_0(\Delta^{m-n-2})-
        (\opn{rank} \opn{Im} d_{\K,n+2}+
        \opn{rank}\opn{Im}\partial_0)\\
        = & \ \opn{rank} \, (\ker d_{\K,n+1}/\opn{Im} d_{\K,n+2})+1.
    \end{split}
\end{equation}
By \eqref{eq: k>=n+4}, \eqref{eq: rank ker Im 2 case 1} and \eqref{eq: rank ker Im 2 case 2}, we have
$$\opn{rank} H\!H^n(\ZZ_{\lambda\K})=\opn{rank} H\!H^n
(\ZZ_{\K})+1.$$
\end{proof}

\subsection*{Proof of Theorem \ref{thm: change of rank}} 
By Corollary \ref{cor: rank H*}, Proposition \ref{prop: rank H n-1}, and Proposition \ref{prop: rank H n}, we complete the proof of Theorem \ref{thm: change of rank}.

\section{The proof of Theorem \ref{thm: complex with any even rank}}
\label{sec: the proof of thm 2}

Theorem \ref{thm: change of rank} becomes particularly concrete when $n=1$ (gluing a $1$-simplex, i.e., an edge, to $\K$). We first specialize it in this case, since it will be the key tool for constructing simplicial complexes with an arbitrary even double cohomology rank.

\begin{prop}
\label{prop: the change of rank n=1}
Let $\Delta^1$ be the $1$-simplex with vertex set $[2] = \{1, 2\}$, and let $\K$ be a simplicial complex on $[m]$ with $m \ge 3$. Suppose that:
\begin{enumerate}
\item $[2] \notin \K$ (i.e., the edge $\{1, 2\}$ is missing from $\K$);
\item For all $3 \le j \le m$, the edges $\{1, j\}, \{2, j\} \in \K$;
\item For all $J = \{1, 2, j, k\}$  with $3 \le j < k \le m$, $\opn{rank}\, \widetilde{H}^1(\K_J) \le 1$;
\item For all $J = \{1, 2, j, k\}$ with $3 \le j < k \le m$ , either both 
$\{2, j, k\}, \{1, j, k\} \in \K$ or both are not in $\K$.
\end{enumerate}
Let $\lambda_{[2]}\K = \K \cup_{\partial\Delta^1} \Delta^1$ (gluing the edge $\{1, 2\}$ to $\K$). If there exists $J = \{1, 2, j, k\}$ with $3 \le j < k \le m$ such that $\opn{rank}\, \widetilde{H}^1(\K_J) = 1$, then
$$\opn{rank}\, H\!H^*(\ZZ_{\lambda_{[2]}\K}) = \opn{rank}\, H\!H^*(\ZZ_{\K}) - 2.$$
\end{prop}

We now prove Theorem \ref{thm: complex with any even rank} via induction and two key operations: simplicial join (Proposition \ref{prop: simplicial join}) and edge gluing (Proposition \ref{prop: the change of rank n=1}). 

\begin{proof}[Proof of Theorem \ref{thm: complex with any even rank}]
    It suffices to construct a family of simplicial complexes
$\{\K_{2r}:r\ge 1\}$ such that the following two conditions hold for any
$r$:
\begin{enumerate}
    \item 
    $\opn{rank}\, H\!H^*(\ZZ_{\K_{2r}})=2r$.
    \item 
    There  exist two vertices of $\K_{2r}$, say $x_{2r},y_{2r}$, such
    that $\{x_{2r},y_{2r}\}\not\in \K_{2r}$.
\end{enumerate}

Let $\K_4=\{\{1\},\{2\},\{3\},\{4\},\{1,2\},\{2,3\},\{3,4\},\{1,4\}\}$ be the square, as shown in Figure \ref{fig: K 2r}. Then \cite[Example 6.4]{LPSS-2023} shows
$\opn{rank}\, H\!H^*(\ZZ_{\K_4})=4$. Let $x_4=1,y_4=3$.

As shown in Figure \ref{fig: K 2r}, we let $\K_2=\lambda_{\{1,3\}}\K_4$, which is the union of $\K_4$ and the $1$-simplex on $\{1,3\}$ as a set. Note that $\K_4\simeq S^1$. By Proposition \ref{prop: the change of rank n=1}, $\opn{rank}\, H\!H^*(\ZZ_{\K_2})=2$.  And let
$x_2=2,y_2=4$.

For $2r\ge 6$, suppose that we have constructed the simplicial complexes $\K_{2t}$
for $1\le t\le r-1$ such that the above two conditions hold.
Then we divide two cases to construct $\K_{2r}$.

(1) Assume that $r$ is even. Note that $\opn{rank}\, H\!H^*(\ZZ_{\{\{A\},\{B\}\}})=2$
by \cite[Theorem 6.5]{LPSS-2023}, where $\{\{A\},\{B\}\}$ denotes the complex which has only two isolated vertices $A,B$.
Let $\K_{2r}=\K_r*\{\{A\},\{B\}\}$. Then, by Proposition \ref{prop: simplicial join}, 
$$\opn{rank}\, H\!H^*(\ZZ_{\K_{2r}})=2\opn{rank}\, H\!H^*(\ZZ_{\K_r})=2r.$$
Let $x_{2r}=A,y_{2r}=B$ and thus $\{x_{2r},y_{2r}\}\not\in \K_{2r}$.

(2) Assume that $r$ is odd.
Let $\K_{2r}=\lambda_{\{A,B\}}(\K_{r+1}*\{\{A\},\{B\}\})$.
We denote $\LL=\K_{r+1}*\{\{A\},\{B\}\}$. 

For any two vertices $a,b$ of $\K_{r+1}$,
by the definition
of simplicial join, if $\{a,b\}\not\in \K_{r+1}$, then $\LL_{\{a,b,A,B\}}$
is a square. If $\{a,b\}\in \K_{r+1}$, then $\{a,b,A\},\{a,b,B\}\in \LL_{\{a,b,A,B\}}$. Thus, $\LL_{\{a,b,A,B\}}$ is contractible.
Then we have
\begin{equation*}
    \begin{split}
        \LL_{\{a,b,A,B\}}\simeq
        \begin{cases}
            S^1, & \text{ if } \{a,b\}\not\in \K_{r+1},\\
            pt, & \text{ if } \{a,b\}\in \K_{r+1}.
        \end{cases}
    \end{split}
\end{equation*}
In particular, $\LL_{\{x_{r+1},y_{r+1},A,B\}}\simeq S^1$. By Proposition \ref{prop: the change of rank n=1},
$$\opn{rank}\, H\!H^*(\ZZ_{\K_{2r}})=2\opn{rank}\, H\!H^*
(\ZZ_{\K_{r+1}})-2=2r.$$
Let $x_{2r}=x_{r+1},y_{2r}=y_{r+1}$. Then 
$\{x_{r+1},y_{r+1}\}\not\in \K_{r+1}$ implies
$\{x_{2r},y_{2r}\}\not\in \K_{2r}$.

\begin{figure}
    \centering
    \begin{tikzpicture}
    \fill (-2.5,0) circle (2pt);
    \fill (-0.5,0) circle (2pt);
    \fill (-0.5,2) circle (2pt);
    \fill (-2.5,2) circle (2pt);
    \draw[-] (-2.5,0) node [below left] {$1$}--(-0.5,0) node [below right] {$2$}--(-0.5,2) node [right] {$3$}--(-2.5,2) node [left] {$4$}--(-2.5,0);
    \draw[-] (-2.5,0) --(-0.5,2);
    \fill (-1.5,-0.27) node [below] {$\K_2=\lambda_{\{1,3\}}\K_4$} circle (0.1pt);
    \fill (0.5,0) circle (2pt);
    \fill (2.5,0) circle (2pt);
    \fill (2.5,2) circle (2pt);
    \fill (0.5,2) circle (2pt);
    \draw[-] (0.5,0) node [below left] {$1$}--(2.5,0) node [below right] {$2$}--(2.5,2) node [right] {$3$}--(0.5,2) node [left] {$4$}--(0.5,0);
    \fill (1.5,-0.27) node [below] {$\K_4$} circle (0.1pt);
    \fill (3.5,1.25) node [below left] {$1$} circle (2pt);
    \fill (4,1.75) circle (2pt);
    \fill (3.9,1.75) node [above left] {$4$} circle (0.1pt);
    \fill (5,1.25) node [below right] {$2$} circle (2pt);
    \fill (5.5,1.75) node [above right] {$3$} circle (2pt);
    \fill (4.5,3) node [above left] {$A$} circle (2pt);
    \fill (4.5,0) node [right] {$B$} circle (2pt);
    \draw[-] (3.5,1.25)--(5,1.25)--(5.5,1.75)--(4.5,3)--(3.5,1.25)--(4.5,0)--(5,1.25)--(4.5,3);
    \draw[-] (4.5,0)--(5.5,1.75);
    \draw[dashed] (5.5,1.75)--(4,1.75)--(3.5,1.25);
    \draw[dashed] (4.5,3)--(4,1.75)--(4.5,0)--(4.5,3);
    \fill[gray, opacity=0.1] (3.5,1.25) -- (4,1.75) -- (4.5,3) -- cycle;
    \fill[gray, opacity=0.1] (3.5,1.25) -- (4,1.75) -- (4.5,0) -- cycle;
    \fill[gray, opacity=0.1] (3.5,1.25) -- (5,1.25) -- (4.5,3) -- cycle;
    \fill[gray, opacity=0.1] (3.5,1.25) -- (5,1.25) -- (4.5,0) -- cycle;
    \fill[gray, opacity=0.1] (5.5,1.75) -- (5,1.25) -- (4.5,3) -- cycle;
    \fill[gray, opacity=0.1] (5.5,1.75) -- (5,1.25) -- (4.5,0) -- cycle;
    \fill[gray, opacity=0.1] (5.5,1.75) -- (4,1.75) -- (4.5,3) -- cycle;
    \fill[gray, opacity=0.1] (5.5,1.75) -- (4,1.75) -- (4.5,0) -- cycle;
    \fill (4.4,-0.27) node [below] {$\K_6=\lambda_{\{A,B\}}\K_8$} circle (0.1pt);
    \fill (6.5,1.25) node [below left] {$1$} circle (2pt);
    \fill (7,1.75) circle (2pt);
    \fill (6.9,1.75) node [above left] {$4$} circle (0.1pt);
    \fill (8,1.25) node [below right] {$2$} circle (2pt);
    \fill (8.5,1.75) node [above right] {$3$} circle (2pt);
    \fill (7.5,3) node [above left] {$A$} circle (2pt);
    \fill (7.5,0) node [right] {$B$} circle (2pt);
    \draw[-] (6.5,1.25)--(8,1.25)--(8.5,1.75)--(7.5,3)--(6.5,1.25)--(7.5,0)--(8,1.25)--(7.5,3);
    \draw[-] (7.5,0)--(8.5,1.75);
    \draw[dashed] (8.5,1.75)--(7,1.75)--(6.5,1.25);
    \draw[dashed] (7.5,3)--(7,1.75)--(7.5,0);
    \fill[gray, opacity=0.1] (6.5,1.25) -- (7,1.75) -- (7.5,3) -- cycle;
    \fill[gray, opacity=0.1] (6.5,1.25) -- (7,1.75) -- (7.5,0) -- cycle;
    \fill[gray, opacity=0.1] (6.5,1.25) -- (8,1.25) -- (7.5,3) -- cycle;
    \fill[gray, opacity=0.1] (6.5,1.25) -- (8,1.25) -- (7.5,0) -- cycle;
    \fill[gray, opacity=0.1] (8.5,1.75) -- (8,1.25) -- (7.5,3) -- cycle;
    \fill[gray, opacity=0.1] (8.5,1.75) -- (8,1.25) -- (7.5,0) -- cycle;
    \fill[gray, opacity=0.1] (8.5,1.75) -- (7,1.75) -- (7.5,3) -- cycle;
    \fill[gray, opacity=0.1] (8.5,1.75) -- (7,1.75) -- (7.5,0) -- cycle;
    \fill (7.7,-0.27) node [below] {$\K_8=\K_4*\{\{A\},\{B\}\}$} circle (0.1pt);
    \end{tikzpicture}
    \caption{The simplicial complexes $\K_{2r}$ for $r\le 4$}
    \label{fig: K 2r}
\end{figure}

\end{proof}

\bibliographystyle{plain}
\bibliography{Ref}

\end{document}